\documentclass[12pt]{article}
\usepackage{epsfig}
\usepackage{amssymb}
\usepackage{graphicx}
\usepackage[usenames,dvipsnames]{color}
\newcommand\C{{\cal C}}

\renewcommand\O{{\cal O}}
\renewcommand\P{{\cal P}}
\renewcommand\S{{\cal S}}

\newcommand\K{{\cal K}}

\newcommand\PG{{\mbox{\textrm PG}}}
\newcommand\PGL{{\mbox{PGL}}}

\newcommand\GF{{\mbox{GF}}}

\newcommand\V{{\cal V}}

\newcommand\D{{\cal D}}

\newcommand\B{{\cal B}}
\newcommand\R{{\cal R}}
\newcommand\X{{\cal X}}
\newcommand\N{{\cal N}}

\newcommand{\li}{\ell_\infty}
\newcommand{\si}{\Sigma_\infty}

 \newcommand{\Label}{\label}
 \newtheorem{theorem}{Theorem}
\newtheorem{thm}[theorem]{Theorem}
\newtheorem{result}[theorem]{Result}
\newtheorem{lemma}[theorem]{Lemma}
\newtheorem{corollary}[theorem]{Corollary}
\newtheorem{defn}[theorem]{Definition}

\newenvironment{proof}{\noindent{\bfseries Proof}\hspace{0.5em}}{ \null  \hfill $\square$ \par}

\newcommand{\Fq}{\mathbb F_q}

\begin{document}

\title{A characterisation of Baer subplanes}

\author{S.G. Barwick and Wen-Ai Jackson
\\ School of Mathematical Sciences, University of Adelaide\\
Adelaide 5005, Australia
}

\maketitle
 


\begin{abstract} Let $\K$ be a set of $q^2+2q+1$ points in $\PG(4,q)$. We show that if every 3-space meets $\K$ in either one, two or three lines,  a line and a non-degenerate conic, or a twisted cubic, then $\K$  is a ruled cubic surface. Moreover, $\K$ corresponds  via the Bruck-Bose representation   to a tangent Baer subplane of $\PG(2,q^2)$. We use this to prove a characterisation  in $\PG(2,q^2)$ of a set of points $\B$ as a tangent Baer subplane by looking at the intersections of $\B$ with Baer-pencils. 
\end{abstract}

\section{Introduction}

A ruled cubic surface in $\PG(4,q)$ is a variety $\V^3_2$ which rules a line and a conic according to a projectivity (see Section~\ref{2.1} for more details). By \cite[Lemma 2.1]{quinn-conic}, a hyperplane of $\PG(4,q)$ meets a ruled cubic surface in either one, two or three lines, a line and a non-degenerate conic, or a twisted cubic. We show that the converse holds, and prove the following characterisation. 

\begin{thm}\Label{char1} Let $\K$ be a set of $q^2+2q+1$ points in $\PG(4,q)$, $q\geq 3$. Then $\K$ is a ruled cubic surface if and only if   every 3-space meets $
\K$ in either one, two or three lines, a line and a non-degenerate conic, or a twisted cubic. 
\end{thm}

Using known results on the Bruck-Bose representation gives the following corollary. 

\begin{corollary}\Label{corBB} Let $\K$ be a set of $q^2+2q+1$ points in $\PG(4,q)$, $q\geq 3$,  such that every 3-space meets $
\K$ in either one, two or three lines, a line and a non-degenerate conic, or a twisted cubic. Then  there is a regular spread $\S$ in a 3-space $\si$ so that in the corresponding Bruck-Bose plane $\P(\S)\cong\PG(2,q^2)$, the set of points corresponding to $\K$ form a Baer subplane of $\PG(2,q^2)$ which is tangent to $\li$. 
\end{corollary}

We  further prove a characterisation in $\PG(2,q^2)$ of a tangent Baer subplane by the intersection types with $\li$-Baer pencils. 
A \emph{Baer pencil}  in $\PG(2,q^2)$ is the cone of $q+1$ lines joining a vertex point $P$ to a  Baer subline base $b$, so contains $q^2(q+1)+1$ points.  Every  line  of $\PG(2,q^2)$ not through the vertex of a Baer pencil meets the Baer pencil in a Baer subline. 
An $\li$-\emph{Baer pencil}  is a Baer pencil containing $\li$, so has vertex in $\li$, and base  $b$ meeting $\li$ in a  point.
We will prove the following characterisation.

\begin{theorem}\Label{inter1} 
Let $\B$ be a set of $q^2+q+1$ points in $\PG(2,q^2)$, $q\geq5$, with $T=\B\cap\li$ a point. Then $\B$ is a Baer subplane if and only if  every $\li$-Baer pencil meets $\B$ in either a point; one or  two Baer sublines; or a non-degenerate $\Fq$-conic. 
 \end{theorem}

Section~\ref{sec2} contains background material. Section~\ref{sec3} contains a proof of Theorem~\ref{char1} and Corollary~\ref{corBB}. Section~\ref{sec4} contains a proof of Theorem~\ref{inter1}.

\section{Background}\Label{sec2}

\subsection{A ruled cubic surface of $\PG(4,q)$}\Label{2.1}

We consider a scroll in $\PG(4,q)$ that rules a line and a non-degenerate conic according to a projectivity. That is, in $\PG(4,q)$, let $\ell$ be a line  and let $\C$ be a non-degenerate conic in a plane $\pi$ with $\ell\cap\pi=\emptyset$. Let $\theta,\phi$ be the non-homogeneous coordinates of $\ell$, $\C$ respectively. That is, without loss of generality, we can write $\ell=\{(1,\theta,\ 0,0,0)\,|\,\theta\in\Fq\cup\{\infty\}\}$ and $\C=\{(0,0,\ 1,\phi,\phi^2)\,|\,\phi\in\Fq\cup\{\infty\}\}$. Let $\sigma$ be a projectivity in $\PGL(2,q)$ that maps $(1,\theta)$ to $(1,\phi)$. 
Consider the set of $q+1$ lines of $\PG(4,q)$ joining a point of $\ell$ to the corresponding point of $\C$ according to $\sigma$. Then the points on these lines form a \emph{scroll ruling a line and a conic according to a projectivity}. 
 The line $\ell$ is called the \emph{line directrix}. The ruling lines are called \emph{generators} and are pairwise disjoint. Further, the line directrix and the generators are the only lines on the scroll. 
The scroll contains $q^2$ conics, called conic directrices, which pairwise meet in a point. Each conic directrix contains one point of each generator and  is disjoint from $\ell$. 
This scroll is a  variety $\V^3_2$ of dimension 2 and order 3, so  is a {\em ruled cubic surface}. For more details, see \cite{BV}, in particular, all ruled cubic surfaces are projectively equivalent.

\subsection{The Bruck-Bose  representation of $\PG(2,q^2)$ in $\PG(4,q)$}

The Bruck-Bose  representation of $\PG(2,q^2)$ in $\PG(4,q)$ was introduced independently by Andr\'e \cite{andr54} and Bruck and Bose \cite{bruc64,bruc66}. 
Let $\si$ be a hyperplane of $\PG(4,q)$ and let $\S$ be a regular line spread
of $\si$. Consider the  incidence
structure $\P(\S)$ with 
 {\em points}  the points of $\PG(4,q)\setminus\si$ and the lines of $\S$;  {\em lines}  the planes of $\PG(4,q)\setminus\si$ that contain
  an element of $\S$, and a line at infinity $\li$ whose points correspond to the lines of $\S$; and {\em incidence} induced by incidence in
  $\PG(4,q)$. Then $\P(\S)\cong\PG(2,q^2)$.  
  
Associated with a regular spread $\S$ in $\PG(3,q)$ are a unique pair of {\em transversal lines} in the quadratic extension $\PG(3,q^2)$. These transversal lines are disjoint from $\PG(3,q)$, and are conjugate with respect to the map $X=(x_0,\ldots,x_3)\mapsto X^q=(x_0^q,\ldots,x_3^q)$. We denote these transversal lines by $g,g^q$. The spread $\S$ is the set of $q^2+1$ lines $XX^q\cap\PG(3,q)$ for $X\in g$, in particular, the points of $\li$ are in 1-1 correspondence with the points of $g$.  If  $ \X$ is a set of points in $\PG(2,q^2)$, then we denote the corresponding set of points  in the Bruck-Bose representation in $\PG(4,q)$  by $[\X]$.  
For more details on this representation, see \cite{UnitalBook}.  In particular, the representation of Baer sublines and Baer subplanes is well known. 
In particular, a Baer subplane of $\PG(2,q^2)$ tangent to $\li$ corresponds to a ruled cubic surface of $\PG(4,q)$. 
The following result characterises the ruled cubic surfaces of $\PG(4,q)$ which correspond to Baer subplanes of $\PG(2,q^2)$.

\begin{result}\Label{ReyCath}\cite{ReyCath} Let $\K$ be a ruled cubic surface in $\PG(4,q)$ with line directrix $b$. Let $\si$ be a 3-space that meets $\K$ in the line $b$ and let $\S$ be a regular spread in $\si$ containing $b$. Then 
$\K$ corresponds via the Bruck-Bose representation  to  a Baer subplane tangent to $\li$ in the Bruck-Bose plane $\P(\S)\cong\PG(2,q^2)$ if and only if in the extension to $\PG(4,q^2)$, $\K$ contains the transversal lines    of the regular spread $\S$. 
\end{result}

%\begin{result}\label{BV}\cite{BV,ReyCath}  Let $\K$ be a ruled cubic surface in $\PG(4,q)$ with line directrix $b$. Let $\si$ be a 3-space that meets $\K$ in exactly the line $b$. Then the $q^2$ planes containing the $q^2$ non-degenerate conics of $\K$ meet $\si$ in lines which together with $b$ form a regular spread $\S$. Further, $\S$ is the unique spread for which $\K$ corresponds to a Baer subplane in the  Bruck-Bose plane $\P(\S)\cong\PG(2,q^2)$.
%\end{result}
%
%
%

\section{Proof of Theorem~\ref{char1}}\Label{sec3}

We begin by labelling the five different types of 3-spaces given in Theorem~\ref{char1}. 

\begin{defn}\Label{T1}
 Let $\K$ be a set of $q^2+2q+1$ points in $\PG(4,q)$. Let $\Pi$ be a 3-space, we say $\Pi$ has type
\begin{itemize}
\item[$T_1$:]  if $\Pi\cap\K$ is 1 line,
\item[$T_2$:]  if $\Pi\cap\K$ is 2  lines,
\item[$T_3$:]   if $\Pi\cap\K$ is 3 lines,
\item[$T_4$:]  if $\Pi\cap\K$ is a line and a non-degenerate conic,
\item[$T_5$:] if $\Pi\cap\K$ is a twisted cubic.

\end{itemize}
\end{defn}

Throughout this section, we let $\K$ be a set of $q^2+2q+1$ points of $\PG(4,q)$, $q\geq 3$, such that every 3-space has  type $T_1$, $T_2$, $T_3$, $T_4$ or $T_5$, that is, $\K$ satisfies the assumptions of Theorem~\ref{char1}.
 
We prove Theorem~\ref{char1} with a series of lemmas, the structure of these lemmas is as follows. In Lemma~\ref{lem:sticks} we show that $\K$ can be partitioned into  $q+1$ mutually skew lines which we call sticks. In Lemma~\ref{lem:base} we show that $\K$ contains a base line which meets each stick. In Lemma~\ref{lem:unique-conic}, we show that $\K$ contains $q^2$ non-degenerate conics.  In Lemma~\ref{lem:projectivity-odd},  we show that the sticks rule the base line and a non-degenerate conic of $\K$ according to a projectivity.  That is, $\K$ is a ruled cubic surface as defined in Section~\ref{2.1}, with line directrix the base line, and generators the sticks. We then use \cite[Theorem 2.1]{ReyCath} to conclude that $\K$ corresponds to a tangent Baer subplane of $\PG(2,q^2)$.

\begin{lemma}\Label{lem:line-conic-not-plane} Let $\K$ be a set of $q^2+2q+1$ points in $\PG(4,q)$  such that every 3-space has type $T_i$, $i\in\{1,\ldots,5\}$. 
Let $\Pi$ be a 3-space of  Type $T_4$, so $\Pi\cap\K$ is a line $\ell$ and a non-degenerate conic $\C$. Then $\ell\cap\C$ is a point and $\langle\ell,\C\rangle=\Pi$.  
\end{lemma}

\begin{proof}
Let $\Pi$ be a 3-space meeting $\K$ in a line $\ell$ and a non-degenerate conic $\C$, and let $\pi$ be the plane containing $\C$. Suppose $\ell$ is contained in $\pi$. Each 3-space containing $\pi$ has type $T_i$ for some $i\in\{1,\ldots,5\}$, so has type $T_4$ and hence contains no further point of $\K$. Every point of $\PG(4,q)$ lies in a 3-space containing $\pi$, so if $\ell\subset\pi$, then 
 $|\K|=|\K\cap\pi|\le 2(q+1)$, a contradiction. Hence $\ell$ meets $\pi$ in a point. Suppose the point $P=\ell\cap\pi$ is not in $\C$. 
As the $q+1$ $3$-spaces through $\pi$, contain a 
 non-degenerate conic of $\K$,  they are of  Type $T_4$.
Hence each 3-space about $\pi$ also contains a line of $\K$, which necessarily contains the point $P$. Therefore the number of points of $\K$ is
$
|\C|+1+(q+1)\times q=q^2+2q+2$, a contradiction. Hence the line $\ell$ meets the non-degenerate conic  $\C$ in a point. 
\end{proof}

\begin{lemma}\Label{contains-arc}
Let $\K$ be a set of $q^2+2q+1$ points in $\PG(4,q)$, $q\geq 3$, such that every 3-space has type $T_i$, $i\in\{1,\ldots,5\}$. 
Then $\K$ contains at least one non-degenerate conic.
\end{lemma}

\begin{proof}
Suppose that $\K$ does not contain a non-degenerate conic, that is, there are no 3-spaces of type $T_4$. We work to a contradiction.
There are four cases for the lines of $\K$: (a) $\K$ contains no lines; (b) $\K$ contains exactly one line; (c) $\K$ contains two lines which meet; (d) $\K$ contains $s\geq2$ lines which are pairwise skew. 

Case (a): suppose  $\K$ does not contain any lines, then every 3-space has type $T_5$.  Let $\Pi$ be a 3-space, so there is a plane $\pi$ contained in $\Pi$ that contains 3 points of $\K$. 
Each 
 $3$-space through $\pi$ has type $T_5$, so $|\K|=3+(q+1) (q-2)=q^2-q+1$, a contradiction. Hence case (a) does not occur.

Case (b): suppose  that $\K$ contains exactly one line $\ell$.  Let $P\in\K\setminus\ell$ and consider the plane $\alpha=\langle P,\ell\rangle$.  The 3-spaces through $\alpha$ meet $\K$ in at least $P,\ell$, and so are of type $T_2$ or $T_3$, contradicting our assumption that $\K$ contains exactly one line. 
Hence case (b) does not occur.

Case (c): suppose $\K$ contains two lines $\ell,\ell'$ which meet in a point. So $\pi=\langle \ell,\ell' \rangle$ is a plane. The 3-spaces containing $\pi$ contain at least 2 lines of $\K$, so are of type $T_2$ or $T_3$. 
Suppose  there is a 3-space $\Pi$ of type $T_3$ that contains $\pi$.
If the points of $\Pi\cap\K$ all lie in $\pi$, then
every  3-space containing $\pi$ has type $T_3$ and hence contains no further point of $\K$. So
 $|\K|=|\pi\cap\K|\le 3q+1$, a contradiction as $q\geq 3$.
Hence 
 $\Pi$ contains a line $m$ which is not contained in $\pi$. Consider the point $m\cap\pi$. If $m\cap\pi$  is not in $\ell$ or $\ell'$, then every 3-space that contains $\pi$ contains $\ell,\ell'$ and a further point of $\K$, so has type $T_3$. In this case, the 
 number of points of $\K$ is
$
|\K|= (2q+1)+1+(q+1).q=q^2+3q+2$, a contradiction.  Hence the point $m\cap\pi$ lies in $\ell\cup\ell'$. 
Let $x$ be the number of 3-spaces  of type $T_2$ containing $\pi$ and let $y$ be the number of 3-spaces  of type $T_3$ containing $\pi$. As every 3-space containing $\pi$ has type $T_2$ or $T_3$, $x+y=q+1$.  Further, counting the points of $\K$ in planes about $\pi$ we have $2q+1+y.q=q^2+2q+1$. Hence $y=q$ and $x=1$.

 Hence the points of $\K$ consist of the points on $q+2$ lines $\{\ell,\ell',m_1,\ldots,m_q\}$ where each $m_i$ meets $\pi$ in a point of $\ell\cup\ell'$, and no two $m_i$ lies in a common 3-space about $\pi$. Suppose the lines $m_1,\ldots,m_q$ are concurrent at the point $\ell\cap \ell'$.  Let $\Pi$ be a 3-space not through $\ell\cap \ell'$, so 
$\Pi$ meets each of the $q+2$ lines $\{m_1,\ldots,m_q,\ell,\ell'\}$ in distinct points, and $|\K\cap\Pi|=q+2$, contradicting $\Pi$ having type $T_i$ for some $i\in\{1,\ldots,5\}$.  Thus at least one of the   lines $m_1,\ldots,m_q$ meets $\pi$ in a point of $\ell\cup\ell'$ which is not the point $\ell\cap\ell'$.

Without loss of generality, suppose $m_1$ meets $\pi$ in a point of $\ell$ distinct from $\ell\cap \ell'$. Let $X\in\K$ with $X\notin\pi$ and $X\notin m_1$. Label the points on $\ell'$ by $Y_0=\ell'\cap\ell,Y_1,\ldots,Y_q$. So for $i=1,\ldots,q$, the line $Y_iX$ is disjoint from $m_1$.  
Consider the 3-spaces $\Pi_i=\langle Y_i,X,m_1\rangle$, $i=0,\ldots,q$. 
As there are no type $T_4$ 3-spaces; and $\Pi_i\cap\K$ contains the line $m_1$ and the two points $X$, $Y_i$; each $\Pi_i$ has type $T_2$ or $T_3$. 
Now $\Pi_0\cap\K$ contains $\ell,m_1, X$, so $\Pi_0$ has type $T_3$. Hence $\Pi_0$  contains at least $q-1$  points of $\K$ not in $\ell\cup\ell'\cup m_1\cup X$.  
If some $\Pi_i$  has type $T_2$, then the two lines in $\Pi_i\cap\K$ are $XY_i$ and $m_1$. As these two lines are disjoint, $\Pi_i\cap\K$ contains $q-1$ points  not in $\ell\cup\ell'\cup m_1\cup X$. 
If some $\Pi_i$, $i\neq0$,  has type $T_3$, then the three lines in $\Pi_i\cap\K$ are either $m_1,XY_i$ and a further line; or $m_1$ and distinct lines $m' $ through $X$ and  $m''$ through $Y_i$. In the latter case, if $m'$, $m''$ meet, then they meet in a point of $m_1$.   As $XY_i$ and $m_1$ are disjoint, $\Pi_i\cap\K$ contains at least $2q-2$ points not in $\ell\cup\ell'\cup m_1\cup X$. 
Let $a$ be the number of type $T_2$ 3-spaces in $\{\Pi_1,\ldots,\Pi_q\}$ and let $b$ be the number of type $T_3$ 3-spaces in $\{\Pi_1,\ldots,\Pi_q\}$, so $a+b=q$.  
Counting the points of $\K$ gives $|\K|\geq (q-1)+a(q-1)+b(2q-2)+(2q+1+q+1)$. Simplifying using $a+b=q$ gives $0\geq b(q-1)+q+2$, contradicting $b\geq 0$.  
Hence case 
 (c) does not occur.

Case (d): suppose $\K$ contains at least two lines, and the lines contained in $\K$ are pairwise skew.  Let $\ell,m$ be two skew lines of $\K$, then the 3-space $\langle \ell,m\rangle$ has type $T_2$ or $T_3$. 
Let $X\in \ell$ and consider the plane $\pi=\langle X,m\rangle$.  
By assumption, the lines of $\K$ are pairwise skew, so there is a unique line of $\K$ through $X$, namely $\ell$. Let $\Sigma$ be a  3-space through $\pi$, with $\ell$ not in $\Sigma$. As $\Sigma\cap\K$ contains the line $m$ and another point $X$, $\Sigma$ has type $T_2$ or type $T_3$. Hence $\Sigma\cap\K$ contains a line through $X$, so contains the unique line $\ell$ of $\K$ through $X$, a contradiction.   
That is, case (d) does not occur. 

We have shown that if $\K$ does not contain a non-degenerate conic, then none of the cases (a), (b), (c), (d) can occur, a contradiction. Thus $\K$ contains at least one non-degenerate conic.  
\end{proof}

\begin{lemma}\Label{lem:sticks}
Let $\K$ be a set of $q^2+2q+1$ points in $\PG(4,q)$, $q\geq 3$,  such that every 3-space has type $T_i$, $i\in\{1,\ldots,5\}$. 
Then $\K$ can be partitioned into a set of $q+1$ lines called {\em sticks}.
Further, no three sticks lie in a 3-space.
\end{lemma}

\begin{proof}
By Lemma~\ref{contains-arc}, $\K$ contains a non-degenerate conic $\C$. Denote the plane  containing $\C$ by  $\pi$.  Each 3-space which contains $\pi$ has type $T_4$ and so contains a further line of $\K$ meeting $\C$.  Moreover, by Lemma~\ref{lem:line-conic-not-plane}, this line meets $\C$ and is not contained in $\pi$.  Thus the  $q+1$ 3-spaces containing $\pi$  give rise to a set of $q+1$ distinct lines of $\K$, we call these lines \emph{sticks}.  Suppose  that two sticks $m,m'$  meet in a point, which is necessarily a point of $\C$. As the 3-spaces containing $\pi$ partition the points of $\PG(4,q)\setminus\pi$,  there is a point $X\in\C$ through which there is no line of $\K$.  Consider the 3-space $\Pi=\langle m,m',X\rangle$. As $\Pi$ contains two lines and a further point of $\K$, $\Pi$ has type $T_3$ and so contains a line of $\K$ through $X$, a contradiction.  Thus the $q+1$ sticks meets $\C$ in distinct points, and lie in distinct 3-spaces containing $\pi$, hence they are pairwise skew, and so partition the $(q+1)^2$ points of $\K$.

Finally suppose $\Pi$ is a 3-space containing three sticks $m_1,m_2,m_3$. Then $\Pi$ has type $T_3$, so does not contain a non-degenerate conic, so does not contain $\pi$. So $\Pi$ meets $\pi$ in a line $\ell$. Further, each stick meets $\pi$ in a point of $\C$, so $\ell$ contains the three collinear points $m_i\cap\C$, $i=1,2,3$, a contradiction.
\end{proof}

\begin{lemma}\Label{lem:base}
Let $\K$ be a set of $q^2+2q+1$ points in $\PG(4,q)$, $q\geq 3$,  such that every 3-space has type $T_i$, $i\in\{1,\ldots,5\}$. 
Then $\K$ contains exactly $q+2$ lines; consisting of the $q+1$ sticks, and a  \emph{baseline} which intersects every stick.
\end{lemma}

\begin{proof}
By Lemma~\ref{lem:sticks}, $\K$ contains at least $q+1$ lines, namely the sticks. 
Suppose $\K$ contains no further line. Let $m,m'$ be two sticks and let $X$ be any point on $m$.  Consider the plane $\pi=\langle X,m'\rangle$.  
Let $\Pi$ be a 3-space containing $\pi$, as $\Pi$  meets $\K$ in at least a line and a point, $\Pi$ must be type $T_2$, $T_3$ or $T_4$. 
If $\Pi$ is type $T_3$, then 
$\Pi$  contains three lines, which by assumption are all sticks, contradicting Lemma~\ref{lem:sticks}. As there is a unique line of $\K$ through $X$, there is a unique 3-space of type $T_2$ containing $\pi$. 
So the remaining $q$ 3-spaces containing $\pi$ have type $T_4$. Recalling Lemma~\ref{lem:line-conic-not-plane}, such a 3-space meets $\K$ in $2q+1$ points, $q+2$ of which lie in $\pi$. Hence counting the points of $\K$ yields $(q+2)+q+q(q-1)=q^2+q+2$, a contradiction. 
 Hence $\K$ contains at least one line which is not a stick.

Now suppose $\K$ contained two  lines, $\ell,\ell'$ which are not sticks. As the sticks partition $\K$, and $\ell,\ell'$ lie in $\K$ but are not sticks, they each meet every stick in a point. If $\ell,\ell'$ do not meet, then 
 the $q+1$ sticks all  lie in the subspace $\Sigma=\langle \ell,\ell'\rangle$.
  If $\ell,\ell'$   meet in a point, then the $q$ sticks not containing the point $\ell\cap\ell'$ lie  in the subspace $\Sigma=\langle \ell,\ell'\rangle$.
In both cases, $\Sigma$ is  either   a 3-space, or is contained in a 3-space, that does not have type $T_i$, $i=1,\ldots,5$, a contradiction. Hence there is at most one line of $\K$ which is not a stick. Thus $\K$ contains exactly one line which is not a stick, and this line meets every stick. 
\end{proof}

\begin{lemma}\Label{lem:unique-conic}
Let $\K$ be a set of $q^2+2q+1$ points in $\PG(4,q)$, $q\geq 3$,  such that every 3-space has type $T_i$, $i\in\{1,\ldots,5\}$. 
Let $P,Q$ be any two points on distinct sticks, not lying on the baseline of $\K$.  Then there exists a unique non-degenerate conic $\C$ contained in $\K$ through $P,Q$. Moreover, $\C$ meets each stick of $\K$ in a unique point, and does not meet the baseline. 
\end{lemma}

\begin{proof} 
If there were two non-degenerate conics of $\K$ containing two points $P$ and $Q$ of $\K$, then the two conics lie in a 3-space which is not type $T_i$, $i\in\{1,\ldots,5\}$,   a contradiction. Hence $P,Q$ lie in at most one   non-degenerate conic of $\K$.

Let $m,m'$ be sticks of $\K$, and denote the baseline of $\K$ by $b$ (recall the baseline  meets every stick). 
Let $P\in m$, $Q\in m'$, with $P,Q\notin b$. We construct a non-degenerate conic containing $P$ and $Q$. 
Consider the plane $\pi=\langle m,Q\rangle$.  As $\pi$ contains at least the  line $m$ and  point $Q$ of $\K$, 
each 3-space containing $\pi$ is type $T_2$, $T_3$ or $T_4$.
The 3-space $\Pi_0=\langle\pi,b\rangle$ contains $m,m',b$ and so has type $T_3$. As the sticks partition the points of $\K$, there are no type $T_2$ 3-spaces containing $\pi$, and a unique type $T_3$ 3-space containing $\pi$, namely $\Pi_0$. 
 So the $q$ remaining 3-spaces $\Pi_1,\ldots,\Pi_q$ through $\pi$  are type $T_4$. That is, $\Pi_i$, $i=1,\ldots,q$,  meets $\K$ in the stick $m$ and a non-degenerate conic $\C_i$ which contains $Q$. Further, by Lemma~\ref{lem:line-conic-not-plane}, $\C_i$ meets the stick $m$ in a point, but $\C_i$ is not contained in $\pi$. 

Suppose for some $i\neq j$, the non-degenerate conics $\C_i,\C_j$ meet $m$ in the same point $X$, then 
$\C_i,\C_j$ both contain the points $P,X$, contradicting the first paragraph.
Hence each $\C_i$ meets $m$ in a distinct point. Now suppose some $\C_i$ contains the point $m\cap b$, 
Let $\Pi'$ be the 3-space containing $\C_i$ and $b$. As  $\Pi'$ contains two points of the stick $m'$ (namely $m'\cap b$ and $Q\in\C_i$), $\Pi'$ contains $m'$. 
Hence $\Pi'$ is a 3-space and $\Pi'\cap\K$ contains two lines $b,m'$ and a non-degenerate conic $\C_i$, 
 a contradiction. 
Hence each $\C_i$ meets $m$ in a distinct point which is not $m\cap b$. 
One of these points of $m$ is $P$, so one of the $\C_i$ contains both $P$ and $Q$.
That is, $P,Q$ lie in at least one non-degenerate conic of $\K$. We showed above that $P,Q$ lie in at most one non-degenerate conic of $\K$, hence  they lie in a unique non-degenerate conic of $\K$. 
\end{proof}

\begin{lemma}\Label{lem:projectivity-odd}
Let $\K$ be a set of $q^2+2q+1$ points in $\PG(4,q)$, $q\geq 3$,  such that every 3-space has type $T_i$, $i\in\{1,\ldots,5\}$. 
Then $\K$ is a scroll ruling the base line $b$ and a non-degenerate conic according to a projectivity. 
\end{lemma}

\begin{proof}
Let $\K$ be a set of $q^2+2q+1$ points in $\PG(4,q)$, $q\geq 3$,  such that every 3-space has type $T_i$, $i\in\{1,\ldots,5\}$. 
By Lemma~\ref{lem:sticks}, $\K$ contains $q+1$ sticks which we denote $\{m_0,\ldots,m_q\}$. These sticks partition the points of $\K$ and are no three in a 3-space.  By Lemma~\ref{lem:base}, $\K$ contains a base line $b$ which meets each stick in a distinct point. By Lemma~\ref{lem:unique-conic}, $\K$ contains a non-degenerate conic $\C$ that meets each stick in a point, and does not meet $b$. So the points of $\K$ lie on $q+1$ lines $\{m_0,\ldots,m_q\}$ that are pairwise skew,  each joining a point of $b$ to a point of $\C$. That is, $\K$ is a set of lines that rules $b$ and $\C$. We show that this ruling is a projectivity by: in step 1 we show 
that we can project $\K$ from a point of $\K$ onto a hyperbolic quadric in a 3-space; then in step 2 we use this hyperbolic quadric to demonstrate the required projectivity.

Step 1:\ 
Let $P$ be a point on the stick $m_0$ with $P\notin b$. 
By Lemma~\ref{lem:unique-conic}, there is a unique non-degenerate conic of $\K$ through $P$ and any point of $\K$ not on $m_0$ or $b$; moreover, this conic meets $m_0$ in $P$ and does not meet $b$. 
As $|\K|=q^2+2q+1$, we obtain a set $\{\C_1,\ldots,\C_q\}$ of non-degenerate conics of $\K$ which each contain $P$, such that each point in 
 $\K\backslash\{ m_0\cup b\}$ lies in exactly one of these non-degenerate conics. 
 Denote the plane containing $\C_i$ by $\pi_i$. 
 If $\pi_i$ contains a point $X\in\K$ with $X\notin\C_i$, then $\pi_i$ contains the stick through $X$, contradicting Lemma~\ref{lem:line-conic-not-plane}. So $\pi_i\cap\K=\C_i$, $i=1,\ldots,q$, so in particular, $b\cap\pi_i=\emptyset$. 
 Let $t_1$ be any line of $\pi_1$ not through $P$, so $t_1$ is not a line of $\K$ and does not meet $b$. Consider  the 3-space $\Pi=\langle b,t_1\rangle$. 
 If $P\in\Pi$, then $\Pi$ contains the plane $\pi_1$, and so $b$ meets $\pi_1$, a contradiction, hence $P\notin\Pi$.

Let $\phi$ denote the projection of $\K$ from  $P$ onto $\Pi$. We have $\phi(b)=b$. The stick $m_0$ contains $P$ and so $\phi(m_0)$ is    the point $\phi(m_0)=M_0=m_0\cap b$. For $j=1,\ldots,q$, the stick $m_j$ is projected onto a line $s_j=\phi(m_j)$.
For $j=1,\ldots,q$, let $M_j=m_j\cap b$, and note that for $i\neq j$, we have $M_i\neq M_j$. Further, let $P_j=\C_1\cap m_j$ and $Q_j=PP_j\cap t_1$, so we have $s_j=M_jQ_j$. 

Suppose two lines $s_i,s_j$, $i\neq j$ meet in a point $Y$. As $M_i\neq M_j$, we have $Y\notin b$. If $Y\in t_1$, then the line $PY$ contains three points of $\C_1$, namely $P,P_i,P_j$, a contradiction as $\C_1$ is non-degenerate. So $Y\notin t_1$, and so $Q_i\neq Q_j$. Hence $\langle s_i,s_j\rangle$ is a plane containing two points of $b$ (namely $M_i,M_j$) and two points of $t_1$ (namely $Q_i,Q_j$). Hence $b,t_1$ meet, a contradiction. 
Thus $s_1,\ldots, s_q$ are pairwise skew lines.

Hence $\phi(\K)$ is a set of $q^2+q+1$ points, namely the points on the lines  $s_1,\ldots, s_q$, and the point $M_0$. 
 Recall that the base line $b$ lies in $\phi(K)$ and contains the point $M_0$ and meets each of $s_1,\ldots,s_q$ in a distinct  point $M_i=s_i\cap b=m_i\cap b$. 

Now consider the $q$-arcs $\C_i\setminus \{P\}$, $i=1,\ldots,q$.
We have $\phi(\C_i\setminus\{P\})=\langle P, \ \C_i\setminus\{P\}\rangle\cap\Pi$, and so  $\phi(\C_i\setminus\{P\})$ is a set of $q$ points on the line 
$t_i=\pi_i\cap\Pi$.  By Lemma~\ref{lem:unique-conic},  $\C_i$, $i=1,\ldots,q$, meets the stick $m_j$, $j=1,\ldots,q$ in point not on $b$. Hence in $\phi(\K)$, the line $t_i$ meets the line $s_j$ in a point not on $b$, $i,j=1,\ldots,q$.  The $q+1$ lines $\{b,t_1,\ldots,t_q\}$ are mutually skew, and contain the $q^2+q+1$ points of $\phi(\K)$, and $q$ additional points, one  on each line $t_i$. The $q$ lines $\{s_1,\ldots,s_q\}$ are mutually skew, and cover the $q^2+q$ points of $\phi(\K)\setminus \{M_0\}$. Hence $\R=\{b,t_1,\ldots,t_q\}$ is a regulus, with opposite regulus $\R'$ containing $s_1,\ldots,s_q$. As $q\geq 3$,  there is  a unique line $s_0$ in $\R'$ through $M_0$ which meets $t_1,\ldots,t_q$. 
 Thus we can project  $\K$ onto $q^2+q+1$ points of the  hyperbolic quadric $\R\cup\R'$.

 Step 2:\ 
 We showed in step 1 that we can project  $\K$ onto $q^2+q+1$ points of the hyperbolic quadric $\R\cup\R'$. 
 For  $i\in\{1,\ldots,q\}$, consider the stick $m_i$ which joins the point $M_i\in b$ with the point $m_i\cap\C_1$. 
 This maps under $\phi$ to the line $\phi(m_i)=s_i$, joining the point $\phi(M_i)=M_i $ with the point $\phi(m_i\cap\C_1)=s_i\cap t_1$. 
 A hyperbolic  quadric is ruled by a projectivity, in particular, it is ruled by a projectivity from $b$ to $t_1$. That is,  there is a projectivity that maps the point $M_i\in b$ to the point $s_i\cap t_1$, $i=0,\ldots,q$. The projection of $\K$ onto $\Pi$ preserves this projectivity, so we have a projectivity mapping the points of $b\setminus\{M_0\}$ to the points of $\C_1\setminus\{P\}$ which maps $m_i\cap b$ to $m_i\cap \C_1$, $i=1,\ldots,q$. This projectivity will also map $M_0$ to $P$. That is, the sticks of $\K$ are lines of a scroll ruling the line $b$ and the non-degenerate conic $\C_1$ 
 according to a projectivity. 
\end{proof}

{\bfseries Proof of Theorem~\ref{char1}}
Let $\K$ be a set of $q^2+2q+1$ points in $\PG(4,q)$, $q\geq 3$, such that every 3-space has type $T_i$, $i\in\{1,\ldots,5\}$. 
By Lemma~\ref{lem:projectivity-odd}, $\K$ is a scroll that rules a  line and a non-degenerate conic according to a projectivity. 
 Hence  by \cite{BV},  $\K$ is a ruled cubic surface.    \hfill$\square$

{\bfseries Proof of Corollary~\ref{corBB}}
Let $\K$ be a set of $q^2+2q+1$ points in $\PG(4,q)$, $q\geq 3$, such that every 3-space has type $T_i$, $i\in\{1,\ldots,5\}$, then by Theorem~\ref{char1}, $\K$ is a ruled cubic surface. 
Let $\si$ be a 3-space that meets $\K$ in exactly the line directrix (baseline) $b$. For each conic 
 $\C_i$, $i=1,\ldots,q^2$ contained in  $\K$, let $\ell_i$ be the line in $\si$ that lies in the plane 
 containing $\C_i$.
By \cite[Theorem 2.1]{ReyCath}, the lines $\ell_i$, $i=1,\ldots,q^2$ together with $b$ form 
   a  regular spread $\S$, moreover   $\K$ corresponds via the Bruck-Bose representation to a tangent Baer subplane of $\P(\S)\cong\PG(2,q^2)$.
   \hfill$\square$

 \section{Proof of Theorem~\ref{inter1}}\Label{sec4}

An $\li$-\emph{Baer pencil} of $\PG(2,q^2)$ has the following nice representation in the Bruck-Bose setting.

 \begin{result} \cite{BJW1}\Label{Bpencil}
 An $\li$-Baer pencil of $\PG(2,q^2)$ with vertex $L$ corresponds in the $\PG(4,q)$ Bruck-Bose  representation
 to a 3-space of $\PG(4,q)$ containing the spread line $[L]$. 
Conversely if $\Pi$ is a $3$-space in $\PG(4,q)$ distinct from $\si$, then $\Pi$ corresponds in $\PG(2,q^2)$ to an $\li$-Baer pencil with vertex corresponding to the unique spread line in  $\Pi$. 
\end{result}

We first show how an $\li$-Baer pencil   meets a  Baer subplane  of $\PG(2,q^2)$ tangent to $\li$. 

\begin{lemma}\Label{thm3RL} Let $\B$ be a  Baer subplane in $\PG(2,q^2)$ tangent to $\li$ at the point $ T=\B\cap\li$. Then an $\li$-Baer pencil with vertex $L\in\li$ meets $\B$  in either:  the point $T$; one or two Baer sublines containing $T$; two Baer sublines with one containing $T$ and  one contained in a line through $L$; or a non-degenerate $\Fq$-conic of $\B$. 
\end{lemma}

\begin{proof}
Let $\B$ be a  Baer subplane in $\PG(2,q^2)$ tangent to $\li$ at the point $ T=\B\cap\li$, so by Result~\ref{ReyCath},  in $\PG(4,q)$, $[\B]$ is a ruled cubic surface with line directrix the spread line $[T]$. By Result~\ref{Bpencil},  an $\li$-Baer pencil with vertex $L$ corresponds to a 3-space that meets $\si$ in the spread line $[L]$. 
By \cite[Lemma 2.1]{quinn-conic}, a hyperplane of $\PG(4,q)$ meets a ruled cubic surface in one of the following: (a) 1 line; (b) 2 lines, (c) 3 lines, (d) a line and a non-degenerate conic, (e) a twisted cubic. In case (a), the line is $[T]$, and so the corresponding $\li$-Baer pencil meets $\B$ in the point $T$. In case (b), the lines are $[T]$ and one generator which meets $[T]$, so the corresponding $\li$-Baer pencil meets $\B$ in a Baer subline through $T$.
 In case (c), the lines are $[T]$ and two generators which meet $[T]$, so the corresponding $\li$-Baer pencil meets $\B$ in two Baer sublines through $T$. In case (d), the line is a generator, and by \cite[Lemma 2.1]{quinn-conic}, the non-degenerate conic corresponds to a Baer subline disjoint from $\li$,  so the corresponding $\li$-Baer pencil meets $\B$ in two Baer sublines, one containing $T$, and one contained in a line through $L$.
In case (e),  by \cite[Theorem 3.1]{quinn-conic}, the twisted cubic corresponds to a non-degenerate conic in $\B$ which contains the point $T$.
\end{proof}
 
 We now prove that the converse holds.

 Let $\K$ be a set of $q^2+2q+1$ points in $\PG(4,q)$. Recall that in Definition~\ref{T1} we defined five types of 3-spaces with respect to $\K$. Here we define three further types of 3-spaces.
 
 \begin{defn} Let $\K$ be a set of $q^2+2q+1$ points in $\PG(4,q)$. Let $\Pi$ be a 3-space, we say $\Pi$ has type
\begin{itemize}
\item[$T_6$:]   if $\Pi\cap\K$ is one point,
\item[$T_7$:]   if $\Pi\cap\K$ is a non-degenerate conic and two lines.
\item[$T_8$:]   if $\Pi\cap\K$ is a line and a twisted cubic.
\end{itemize}
\end{defn}

\begin{lemma}\Label{lem:8-types} Let $\K$ be a set of $q^2+2q+1$ points in $\PG(4,q)$, $q\geq 5$,  such that every 3-space has type $T_i$, $i\in\{1,\ldots,8\}$, then  every 3-space has type $T_i$, $i\in\{1,\ldots,5\}$. 
\end{lemma}

\begin{proof} 
Suppose there is a 3-space $\Pi$ of type $T_6$, and let $P=\Pi\cap\K$.  Let $\pi$ be a plane contained in $\Pi$ with $P\notin\pi$, so $|\pi\cap\K|=0$. Denote the $q+1$ 3-spaces containing $\pi$ by  $\Pi,\Pi_1,\ldots,\Pi_q$. As $|\pi\cap\K|=0$, $\Pi_i\cap\K$ cannot contain a line. Hence each $\Pi_i$, $i=1,\ldots,q$ has type $T_5$ or $T_6$. Let $x$ be the number of type $T_5$ 3-spaces containing $\pi$ and let $y$ be the number of type $T_6$ 3-spaces containing $\pi$, so $y\geq 1$. Then $x+y=q+1$, and counting the points in $\K$ we have $x(q+1)+y=q^2+2q+1$. These equations have   solution $x=q+1$, $y=0$, a contradiction, hence there is no 3-space of type $T_6$. 

Suppose there is a 3-space $\Pi$ of type $T_7$, 
so $\Pi\cap\K$ is a non-degenerate conic $\C$ and two lines $\ell,m$. 
Let $\pi$ be the plane of $\Pi$ containing the non-degenerate conic  $\C$. 
If $\pi$ contains both $\ell$ and $m$, then every 3-space containing $\pi$ is of type $T_7$, and contains no further point of $\K$, hence $|\K|\leq (q+1)+(q+1)+q$, a contradiction. If $\pi$ contains $\ell$ and meets $m$ in a point, then every 3-space containing $\pi$ is of type $T_7$, and contains  $q$ points of $\K$ not in $\pi$, so $|\K|\geq (q+1)+(q-1)+q(q+1)$, a contradiction. Hence both  $\ell$ and $m$ meet $\pi$ in a point, and  
  $|\pi\cap\K|=q+1+t$ for some $t\in\{0,1,2\}$.  Denote the $q+1$ 3-spaces containing $\pi$ by  $\Pi,\Pi_1,\ldots,\Pi_q$. As $\pi$ contains a non-degenerate conic of $\K$, each of these 3-spaces has type $T_4$ or $T_7$. 
Let $x$ be the number of 3-spaces containing $\pi$ of type $T_4$, 
$y$  the number of 3-spaces containing $\pi$ of type $T_7$, 
so $x+y=q+1$ and $y\geq 1$. 
A 3-space of type $T_4$ containing $\pi$ contains  $q$ points of $\K$ that are not in $\pi$.  
A 3-space of type $T_7$  containing $\pi$  contains at least  $2q-1$ points of $\K$ that are not in $\pi$.  Counting points of $\K$ gives 
$|\K|\geq xq+y(2q-1)+(q+1+t)$, hence $y(q-1)+t\leq0$. As $q\geq 3$, we have $y=t=0$,  contradicting $y\geq1$. 
Hence there is no 3-space of type $T_7$.

Suppose there is a 3-space $\Pi$ of type $T_8$, so $\Pi\cap\K$ is a line $\ell$ and a twisted cubic $\N_3$.
The line $\ell$ contains at most two points of $\N_3$, and each plane of $\Pi$ contains at most three points of $\N_3$. As $q\geq4$, there are at least three points of $\N_3$ not on $\ell$. 
Let $\pi$ be the plane containing these three points, so $\pi\cap\ell\notin\N_3$. 
 As $\pi\cap\ell$ is a point in $\K$, we have 
 $|\pi\cap\K|=4$. 
 Let $\Sigma$ be any 3-space containing $\pi$. As $|\pi\cap\K|=4$, $\Sigma$ has type $T_7$ or $T_8$. As we have shown that there are no type $T_7$ 3-spaces, each 3-space containing $\pi$ has type $T_8$. Hence $|\Sigma\cap\K|\geq 2q$. 
Counting points of $\K$ in the $q+1$ 3-spaces containing $\pi$ gives 
$(q^2+2q+1)\geq 4+(q+1)(2q- 4)$, and so $q(q-4)\leq 1$, a contradiction as $q\geq5$.
Hence there is no 3-space of type $T_8$. 
\end{proof}

\begin{lemma}\Label{thm3LR}
Let $\B$ be a set of $q^2+q+1$ points in $\PG(2,q^2)$, $q\geq5$, with $T=\B\cap\li$ a point. Suppose every $\li$-Baer pencil meets $\B$ in either a point; one or  two Baer sublines; or a non-degenerate $\Fq$-conic.
Then $\B$ is a Baer subplane. 
\end{lemma}
 
 \begin{proof}
 Let $\B$ be a set of $q^2+q+1$ points in $\PG(2,q^2)$, $q\geq5$, with $T=\B\cap\li$ a point.  
  We  work in the Bruck-Bose representation. So $\PG(2,q^2)$ is represented in $\PG(4,q)$, and the representation  involves a regular 1-spread $\S$ in the hyperplane at infinity $\si$. 
 In the Bruck-Bose representation in $\PG(4,q)$, $[\B]$ is a set of $q^2+q$ affine points and the spread line $[T]$, so $[\B]$ is a set of $q^2+2q+1$ points. 
 
 Let $[\D]$ be any 3-space distinct from $\si$. Then $[\D]$ contains a unique spread line $[P]$. 
By \cite{BJW1}, $[\D]$ corresponds to an $\li$-Baer pencil $\D$ in $\PG(2,q^2)$ with vertex $P$ (possibly $P=T$). 
As $\D$ contains $\li$, $T\in\D$, so $T\in\B\cap\D$. So
by assumption  
 $\B\cap\D$
 is either a point; one or  two Baer sublines; or a non-degenerate conic. 
 As each $\li$-Baer pencil contains $T$, we have four cases: a 
 $\li$-Baer pencil meets $\B$ in either (a) the point $T$; (b) one Baer subline  containing  $T$; (c) two Baer sublines (at least one containing $T$); or (d) a non-degenerate 
 $\Fq$-conic  containing $T$. We show that in each case, $[\B]\cap[\D]$ is type $T_i$ for some $i\in\{1,\ldots,8\}$.

Case (a): suppose $\B\cap\D$ is the point $T$. If $P=T$, then in $\PG(4,q)$, $[\B]$ meets $[\D]$ in the spread line $[T]$, and so   $[\D]$ is a 3-space of type $T_1$. 
 If $P\neq T$, then $[\B]$ meets $[\D]$ in one point of $[T]$, and so   $[\D]$ is a 3-space of type $T_6$.

Case (b): suppose $\B\cap\D$ is a Baer subline $n$ through  $T$. Then in $\PG(4,q)$, $n$ corresponds to a line $[n]$ that meets $\si$ in a point of $[T]$. Hence if $P=T$, then
 $[\B]$ meets $[\D]$ in two lines, namely the spread line $[T]$ and the line $[n]$. That is,  $[\D]$ is a 3-space of type $T_2$. However,  if $P\neq T$, then
 $[\B]$ meets $[\D]$ in one line, namely  $[n]$, and so  $[\D]$ is a 3-space of type $T_1$. 

Case (c): suppose $\B\cap\D$ is two Baer sublines $m,n$, with $T\in m$. 
  The line $m$ contains $T$ and so corresponds to a line $[m]$ of $\PG(4,q)$ that meets $\si$ in a point. If $T\in n$, then $[n]$ is  a line  of $\PG(4,q)$ that meets $\si$ in a point. If $T\notin n$, then $[n]$ is  a conic  of $\PG(4,q)$.  
  Hence if $P\neq T$, then $[\B]$ meets $[\D]$ in either  two lines $[n],[m]$ (and so  $[\D]$ is a 3-space of type $T_2$); or in  a conic and  a line $[m]$ (and so  $[\D]$ is a 3-space of type $T_4$). 
If $P=T$, then
$[\B]$ meets $[\D]$ in: either  three lines, namely the spread line $[T]$ and the lines $[n],[m]$, and $[\D]$ is a 3-space of type $T_3$; 
or in  a conic and  two lines $[T]$, $[m]$, and $[\D]$ is a 3-space of type $T_7$.

Case (d): suppose $\B\cap\D$ is a non-degenerate $\Fq$-conic  $\C$. Note that as $T\in\D$, we have $T\in\C$. By \cite{quinn-conic}, in $\PG(4,q)$, $[\C]$ is a twisted cubic that meets $[T]$ in a point. Hence  in $\PG(4,q)$, if $P=T$,  
 $[\B]$ meets $[\D]$ in a line and a twisted cubic.  That is,  $[\D]$ is a 3-space of type $T_8$. If $P\neq T$, then  $[\B]$ meets $[\D]$ in a  twisted cubic, and so   $[\D]$ is a 3-space of type $T_5$. 

Hence we have shown that  every 3-space of $\PG(4,q)$ has type $T_i$ for some $i\in\{1,\ldots,8\}$. 
So by Lemma~\ref{lem:8-types}, every 3-space has type $T_i$ for some $i\in\{1,\ldots,5\}$. 
Hence by Theorem~\ref{char1}, $[\B]$ is a ruled cubic surface of $\PG(4,q)$. 

By Result~\ref{ReyCath}, $[\B]$ corresponds to a Baer subplane of $\PG(2,q^2)$ if and only if the extension of $[\B]$ to $\PG(4,q^2)$ contains the transversal lines  of the regular spread $\S$.
%
%
%
%As $\B\cap\li=T$, the ruled cubic surface $[\B]$ meets $\si$ in the line directrix $[T]$. 
% Let $[\C_i]$,  $i=1,\ldots,q^2$, be the $q^2$ conics contained in $[\B]$. Let the plane containing $[\C_i]$ meet $\si$ in the line $[\ell_i]$. 
% By \cite[Theorem 2.1]{ReyCath}, $\S'=\{[T],[\ell_1],\ldots,[\ell_{q^2}]\}$ is a
%   regular spread of $\si$, and   $[\B]$ corresponds via the Bruck-Bose representation to a tangent Baer subplane of $\P(\S')$. 
%   We need to show that $\S=\S'$. 
%   
Let $[L]$ be a spread line distinct from $[T]$, and $[\D]$   a 3-space containing $[L]$.
Denote the    transversals lines of $\S$ by $g,g^q$, so the extension of $[L]$ meets the transversal lines $g$ and $g^q$ in points we denote $g\cap [L]$, $g^q\cap [L]$. We need to show that these two points lie in the extension of $[\B]$. 

 By assumption, $[\D]$ does not contain $[T]$. So by  \cite[Lemma 2.1]{quinn-conic},  $[\D]$  meets the ruled cubic surface $[\B]$ in either  (i) a non-degenerate conic disjoint from $[T]$ and a generator line of $[\B]$, or (ii)  a twisted cubic with a point in $[T]$. 
   Correspondingly in $\PG(2,q^2)$, $\D$ is an $\li$-Baer pencil with vertex $L\neq T$.

 First we look at case (i). The generator line in $[\D]\cap[\B]$ corresponds in $\PG(2,q^2)$ to a Baer subline through $T$. So by assumption,  $\D\cap\B$ is two Baer sublines. 
 The conic  in $[\D]\cap[\B]$ 
 is a set of $q+1$ affine points, no two on a line with $T$, and so  corresponds to a Baer subline  $b$ that lies in a line $m$ with $m\cap \li \neq T$.  In $\PG(4,q)$,  $[b]$ is a non-degenerate conic lying in a plane about a spread line. As the only spread line in $[\D]$ is $[L]$, the conic $[b]$ lies in a plane about $[L]$.
    Moreover, the   extension of  the conic $[b]$ to $\PG(4,q^2)$ contains the two points    
    $g\cap [L]$, $g^q\cap [L]$, see \cite{UnitalBook}.   Hence the extension of the ruled cubic surface $[\B]$ contains the two points $g\cap [L]$, $g^q\cap [L]$. 
    
    In case (ii),   the twisted cubic in $[\D]\cap[\B]$ contains a point of $[T]$, so contains $q$ affine points no two on a line through $T$. Hence 
    by assumption,  $\D\cap\B$
is an $\Fq$-conic $\O$ which contains the point  $T$.  By  \cite{BJW1}, $[\O]$ corresponds in $\PG(4,q)$ to a 
 twisted cubic  $[\O]$ whose extension to $\PG(4,q^2)$ contains the two points  $g\cap [L]$, $g^q\cap [L]$. Hence the extension of the ruled cubic surface contains the  two points  $g\cap [L]$, $g^q\cap [L]$. 
 
Hence for each line $[L]$ in the spread $\S$ distinct from $[T]$, the points  $g\cap [L]$, $g^q\cap [L]$ lie in the extension of 
  the ruled cubic surface $[\B]$, and so $g$, $q^q$ are lines of the extended ruled cubic surface. Thus by Result~\ref{ReyCath}, $\B$ is a tangent Baer subplane of $\PG(2,q^2)$, as required. 
\end{proof}
%
%
%\newpage
%
%
%\section{Conclusion}
%
%An interesting question is whether  Theorem~\ref{char1} can be generalised to combinatorial assumptions. That is, 
%suppose $\K$ is a set of $q^2+2q+1$ points in $\PG(4,q)$  such that every 3-space meets $
%\K$ in either one, two or three lines, a line and a $(q+1)$-arc, or a $(q+1)_3$-arc. Is  $\K$  a ruled cubic surface? With this more general assumption, the proofs in this article can be used to show that $\K$ is a scroll ruling a base line and a $(q+1)$-arc. However, the proof of Lemma~\ref{lem:projectivity-odd} needs this $(q+1)$-arc to be a non-degenerate conic, so if $q$ is even, our proof does not construct a projectivity under this more general assumption.
%\newpage
%%%

Theorem~\ref{inter1} now follows immediately from Lemmas~\ref{thm3RL} and~\ref{thm3LR}.

\end{document}